\documentclass[12pt]{amsart}
\usepackage{amssymb}
\newtheorem{thm}{Theorem}
\newtheorem{lem}[thm]{Lemma}

\newtheorem{cor}[thm]{Corollary}

\usepackage{tikz}
\usepackage{a4wide}

\theoremstyle{definition}  

\newtheorem{definition}[thm]{Definition}

\newcommand{\Z}{\mathbb Z}

\newcommand{\N}{\mathbb N}

\usepackage{color}
\usepackage{pgf}
\usepackage{tikz}
\usetikzlibrary{arrows,automata}
\usepackage[latin1]{inputenc}
\usepackage{verbatim}

\usepackage{pgf}
\usepackage{tikz}
\usetikzlibrary{arrows,automata}
\usepackage[latin1]{inputenc}
\usepackage{verbatim}
\usepackage{amssymb}

\begin{document}

\author{Robbert Fokkink and  Reem Yassawi}
\address{TU Delft, Netherlands 
and \\
Institut Camille Jordan, Universit\'{e} Lyon-1, France
}
\email{R.J.Fokkink@tudelft.nl\\yassawi@math.univ-lyon1.fr}


\title[Topological rigidity of linear cellular automaton shifts]
{Topological rigidity of linear cellular automaton shifts}

\begin{abstract}
We prove that topologically isomorphic linear cellular automaton
shifts are algebraically isomorphic. Using this, we show that two distinct such shifts cannot be isomorphic. We conclude  that the automorphism group of a linear cellular automaton shift is a finitely generated abelian group.
\end{abstract}

\maketitle

\section{Introduction}

A full shift consists of a space,  the set of doubly infinite sequences $\{1, 2, \ldots, q\}^{\mathbb Z}$, and the
 transformation $\sigma$ acting on points in that space, defined by $\sigma(x)_n=x_{n+1}$. A multi-dimensional shift is defined analagously, where the space
 $\{1,2,\ldots,q\}^{\mathbb Z^d}$ is acted on by $d$ shifts, defined by
\[
\sigma_i(x)_{(n_1,\ldots,n_i,\ldots,n_d)}=x_{(n_1,\ldots,n_i+1,\ldots,n_d)}
\]
for transformations $\sigma_1,\sigma_2,\ldots,\sigma_d$.
A \textit{subshift} is a closed,  shift-invariant subset of a full shift.
It is \textit{Markov} if it is a shift of finite type. We refer to \cite{Lind-Marcus-1995} for definitions and the basic topological set-up. 

In symbolic dynamics $\{1,\ldots,q\}$ is a finite set with no additional structure. In algebraic dynamics
$\{1,\ldots,q\}$ is a finite abelian group or a finite field.
In this paper, we limit ourselves to the simplest case, when $q$ is prime.
To emphasize this, we write $p$ from now on, instead of $q$, and we denote
the finite field by $\mathbb F_p$.
Thus $\{1,2,\ldots,p\}^{\mathbb Z^d}$ becomes a compact abelian group under coordinatewise addition~$\mathbb F_p^{\mathbb Z^d}$. A \textit{Markov subgroup} is a subshift that is also an additive subgroup of
$\mathbb F_p^{\mathbb Z^d}$. Any such group is a shift of finite type.
A very nice survey of Markov groups is given in~\cite{Kitchens-1997}. 
One motivation of our paper is to try and find topological analogues of the metric results in that survey.
 
The standard example of a Markov subgroup is the \textit{Ledrappier shift} \cite{Ledrappier-1979} 
defined by	
\[
\Lambda=\{(x)_{(m,n)}\colon x_{(m,n)}+x_{(m+1,n)}+x_{(m,n+1)}=0 \text{ for each }i,j\in\mathbb Z^2\}\subset \{0,1\}^{\mathbb Z^2}.
\]
The defining relation of $\Lambda$ corresponds to an $L$-shape in the lattice $\mathbb Z^2$:
\begin{center}
\includegraphics[width=1.5cm]{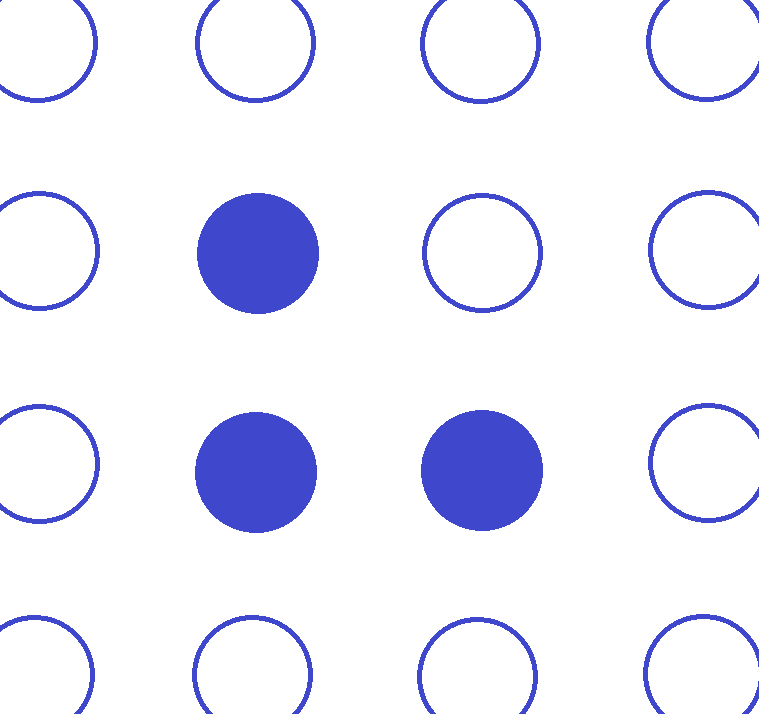}
\end{center}
It is convenient to identify $(c_{(i_1,\ldots,i_d)})\in\mathbb F_p^{\mathbb Z^d}$
with the  formal Laurent series
\[
\sum_{(i_1,\ldots,i_d)\in\mathbb Z^d} c_{i_1,\cdots,i_d}X_1^{i_1}\cdots X_d^{i_d}.
\]
The shift $\sigma_i$ is given by the multiplication by $X_i^{-1}$.
We denote the set of all Laurent series by
$\mathbb F_p[[X_1^{\pm 1},\ldots,X_d^{\pm 1}]]$.
A \textit{Laurent polynomial} is a Laurent series in which all but finitely
many coefficients are zero. We denote the
set of Laurent polynomials by $\mathbb F_p[X_1^{\pm 1},\ldots,X_d^{\pm 1}]$.
It is a unique factorization domain and the set of Laurent series
$\mathbb F_p[[X_1^{\pm 1},\ldots,X_d^{\pm 1}]]$ is a module over this domain.
A Markov subgroup is a subgroup of
$\mathbb F_p[[X_1^{\pm 1},\ldots,X_d^{\pm 1}]]$ which is invariant under
multiplication by $X_i$, i.e., it is a submodule. 
The annihilator of a Markov subgroup $M$
is the ideal of all polynomials $P$ such that $Px=0$ for each $x\in M$.
Annihilator are finitely generated
since $\mathbb F_p[[X_1^{\pm 1},\ldots,X_d^{\pm 1}]]$ is Noetherian.
We will be interested in particular in submodules~$P^{\perp}$ that have an annihilator that is
generated by a single polynomial~$P$. 
For example, with this notation, the Ledrappier shift is equal to $\left(1+X_1^{-1}+X_2^{-1}\right)^\perp$.

Two Markov shifts $M_1$ and $M_2$ are \textit{isomorphic} if there exists
an invertible map $\phi\colon M_1\to M_2$ which is shift commuting, i.e.,
$\phi\circ\sigma_i=\sigma_i\circ\phi$ for all $i$. If $\phi$ is a homeomorphism,
then it is a \textit{topological isomorphism}. If $\phi$ is measure
preserving, then it is a \textit{measurable isomorphism}. 
If $M_1$ and $M_2$ are Markov subgroups and $\phi$ is an isomorphism
between modules, then it is an \textit{algebraic isomorphism}.
An algebraic isomorphism is continuous and preserves  the Haar probability  measure, which 
is the only measure we consider.
 Kitchens conjectured that if $P^{\perp}$ and $Q^{\perp}$
are measurably isomorphic, then they are algebraically isomorphic~\cite{Kitchens-1997}.
This conjecture has been proved for irreducible and strongly mixing $P^{\perp}$ and $Q^{\perp}$:

\begin{thm}[Kitchens-Schmidt, \cite{Kitchens-Schmidt-2000}] Suppose that $P$ and $Q$ are  irreducible elements of $\mathbb F_p[X_1^{\pm 1},\ldots,X_d^{\pm 1}]$. 
If $P^{\perp}$ and $Q^{\perp}$ are measurably isomorphic and strongly mixing,
then they are algebraically isomorphic.
\end{thm}

We are interested in the topological version of Kitchens's conjecture.
We were unable to settle this topological version in full generality,
and restrict our attention to polynomials of the following form:
if \[P=X_d-\Phi(X_1,\ldots,X_{d-1})\]
then we say that $P^{\perp}$ is a 
\textit{linear cellular automaton shift}. Points of a linear cellular automaton shift, that is specific two dimensional configurations in these shifts,   have been studied by  many, for example by Martin, Odlyzko and Wolfram in \cite{Martin-Odlyzko-Wolfram-1984}. 
To our knowledge it is Ledrappier who first studied  a particular such Markov group,  and Kitchens and Schmidt who first studied these shifts in full generality.
We require that $\Phi$ contains at least two non-zero terms, otherwise
the shift would be trivial.
For cellular automaton shifts it is customary to decompose the module
$\mathbb F_p[[X_1^{\pm 1},\ldots,X_d^{\pm 1}]]$ into layers which are indexed by powers
of $X_d$. The powers of $X_d$ form the \textit{time axis}.
One imagines layers to be changing over	time.
An individual time step from one layer to the next
corresponds to a multiplication by $\Phi(X_1,\ldots,X_{d-1})$.

In this article, we show in Theorem \ref{homomorphisms-are-algebraic} that topologically isomorphic linear cellular automaton shifts
are algebraically isomorphic. We apply this result to deduce, in Corollary \ref{homomorphisms-are-trivial}, that distinct linear cellular automaton shifts cannot be topological factors of one another. We also show in Corollary \ref{automorphism-group-is-finitely-generated} that the automorphism group of a linear cellular automaton shift is finitely generated and abelian.

\section{Mixing}\label{mixing}

A probability measure preserving system $(X,\mathcal B, \mu, T)$ is {\em $k+1$-mixing} if
for all measurable sets
$A_0,\ldots,A_k$ 
\[\mu(A_0\cap T^{-n_1}A_1\cap \cdots\cap T^{-n_k}A_k)\longrightarrow\mu(A_0)\cdots\mu(A_k)\] 
as $n_1, n_2-n_1,\ldots,n_k-n_{k-1}\longrightarrow \infty$. 
The term strongly mixing is used to refer to $2$-mixing.
It has been an open problem
for some time, due to Rokhlin~\cite{Rokhlin-1949}, whether strong mixing implies mixing of all orders. 

Let $G$ be a countable abelian group. A $G$-\textit{system} $(X,\mathcal B, \mu, (T_g)_{g\in G})$
is a probability space with a measure preserving $G$-action. It is $k+1$-mixing if
for all measurable sets
$A_0,\ldots,A_k$ 
\[\mu(A_0\cap T_{-g_1} A_1\cap \cdots\cap T_{-g_k} A_k)\longrightarrow\mu(A_0)\cdots\mu(A_k)\] 
as $g_i\longrightarrow\infty$ and $g_j-g_i\longrightarrow\infty$ for all $i\not=j$. 
We often abbreviate the notation and denote a $G$-system simply by $X$. Let $\mu$ be the Haar measure on $P^\perp$ and let $\mathcal B$ be the Borel $\sigma$-algebra.
Ledrappier's shift is a $\mathbb Z^2$-system which is
$2$-mixing but not $3$-mixing.
Another way to define $k+1$-mixing would be that 
\[\mu(T_{-g_0} A_0\cap T_{-g_1} A_1\cap \cdots\cap T_{-g_k} A_k)\longrightarrow\mu(A_0)\cdots\mu(A_k)\] 
as $g_j-g_i\longrightarrow\infty$ for all $i\not=j$. The two definitions are equivalent
because $T_{g_0}$ is a measure preserving automorphism.

In our case the probability space is a Markov shift 
$P^{\perp}\subset \mathbb F_p[[X_1^{\pm 1},\ldots,X_d^{\pm 1}]]$
endowed with the Haar measure and $G$ is equal to $\mathbb Z^d$.
For $\mathbf n\in\Z^d$ the transformation $T_\mathbf n$ is
the group automorphism
\[
\sum c_{i_1,\cdots,i_d}X_1^{i_1}\cdots X_d^{i_d}\ \longrightarrow 
\sum c_{i_1,\cdots,i_d}X^{i_1+n_1}\cdots X^{i_d+n_d}
\]
if $\mathbf n=(n_1,\ldots,n_d)$. In other words, $T_\mathbf n$ is the multiplication by the monomial
with exponent~$\mathbf n$, which we denote by
$
X^\mathbf n=X_1^{n_1}\cdots X_d^{n_d}.
$

It is convenient to describe Markov subgroups as modules, but they are also compact abelian groups.
We recall Halmos's classical result \cite{Halmos-1943},  \cite[Theorems 1.10 and 1.28]{Walters-1982}, which is that 
a continuous automorphism $T$ of a compact abelian group $\Gamma$ is $2$-mixing if and only if the induced automorphism on the character group $\widehat{\Gamma}$ has no finite orbits.
Kitchens and Schmidt finesse this characterisation of 2-mixing for algebraic shifts.
 
\begin{lem}\cite[Proposition 2.11]{Kitchens-Schmidt-1992}\label{2-mixing-in-one-direction}
Let $P=P_1 \ldots P_\ell$ for irreducible Laurent polynomials~$P_i$. 
Then $(P^{\perp}, \mathcal B,\mu, T_{\mathbf n})$ is not 2-mixing if and only if 
one of the factors $P_i$ is a polynomial in $X^\mathbf m$
for some $\mathbf m\in\Z^d$.
\end{lem}
 
\begin{thm}\cite[Theorem 2.4]{Kitchens-Schmidt-1989}\label{2-mixing}
If the algebraic shift  $(P^{\perp}, \mathcal B, \mu, T_\mathbf n)$ is $2$-mixing for every $\mathbf n\in\mathbb Z^d$, 
then the $\mathbb Z^d$-system $(P^{\perp}, \mathcal B, \mu,  \mathbb Z^d)$ is 2-mixing.
 \end{thm}

Lemma \ref{2-mixing-in-one-direction} and Theorem \ref{2-mixing} combined with the following lemma
imply that a cellular automaton shift $P^{\perp}$ is $2$-mixing. 

\begin{lem}
A Laurent polynomial $X_d-\Phi(X_1,\ldots,X_{d-1})\in\mathbb F_p[X_1^{\pm 1},\ldots,X_d^{\pm 1}]$ 
is irreducible.
\end{lem}

\begin{proof}
We denote the ring of Laurent polynomials in $d-1$ variables
by $\mathfrak R_{d-1}$.
The polynomial $X_d+r$ has degree one and therefore is irreducible
in $\mathfrak R_{d-1}[X_d]$ for any $r\in \mathfrak R_{d-1}$. 
It remains irreducible in the localization of $\mathfrak R_{d-1}[X_d]$ by 
the multiplicative set $S=\{X_d^n\colon n\in\mathbb N\}$.
The ring of fractions $S^{-1}\mathfrak R_{d-1}$ is equal to~$\mathfrak R_d$.
Therefore $X_d+r$ is irreducible in $\mathfrak R_d$.
\end{proof}

If a monomial $X^\mathbf n$ has a non-zero coefficient in $P$, then we
say that it \emph{occurs} in~$P$.
We say that the set 
\[
S(P)=\{\mathbf n\colon X^\mathbf n \text{ occurs in }P\}\]
is the \textit{shape} of the polynomial. 
For instance, the shape of the Ledrappier polynomial is equal to
$\{(0,0), (-1,0), (0,-1)\}$.

We say that a finite subset $S=\{\mathbf n_0,\ldots,\mathbf n_k\}\subset \mathbb Z^d$ is \textit{mixing} (for the particular action we are considering) if for all measurable sets $A_0,\ldots,A_k$
\[\mu(T_{-m\mathbf n_0}A_0\cap T_{-m\mathbf n_1} A_1\cap \cdots\cap T_{-m\mathbf n_k} A_k)\longrightarrow\mu(A_0)\cdots\mu(A_k)\] 
as $m\longrightarrow\infty$.
Note that $S$ is mixing if and only if any of its translates $\mathbf n+S$ is mixing,
since $T_\mathbf n$ is a measure preserving automorphism.
Therefore, we may translate $S$ so that it contains~$0$. 
We say that $S$ is \textit{primitive} if $0\in S$. Similarly,  
we say that $P$ is \textit{primitive} if $0$ occurs in $P$.

If a $\mathbb Z^d$-system is $k+1$-mixing, then all primitive sets
of cardinality $k+1$ are mixing. The converse is also true but this is not obvious. 
In fact, this remained an open problem for quite some time, until it was solved by Masser~\cite{Masser-2004}.
Ledrappier's shift is $2$-mixing but not mixing on $S=\{(0,0),(-1,0),(0,-1)\}$, as follows
from the following lemma.

\begin{lem}
$S(P)$ is a non-mixing set of $P^{\perp}$. 
More generally, $S(Q)$ is non-mixing
for any $Q\in \langle P\rangle$.
\end{lem}

\begin{proof}
Without loss of generality we may assume that $P$ is primitive.
Consider the cylinder set 
$A$ of all Laurent series $\sum_{\mathbf n} c_{\mathbf n}X^{\mathbf n}$
which have constant coefficient $c_{\mathbf 0}=1$.
Similarly, let $B$ be the cylinder set of Laurent series 
in $P^{\perp}$ such that $c_{\mathbf 0}=0$.
Let $\{\mathbf 0, \mathbf n_1,\ldots,\mathbf n_k\}$ be the shape of $P$.
The elements of
\[
A\cap T_{-\mathbf n_1} B\cap \cdots\cap T_{-\mathbf n_k} B
\]
are exactly those Laurent series $L$ which have constant coefficient $1$ and
coefficient zero at all other monomials $X^{-\mathbf n}$ which occur in $P$. 
Therefore $P\cdot L$ has constant coefficient $1$.
In particular, $L$ is not annihilated by $P$. 

Observe that the shape of $P^p$ is equal to $pS(P)$ and more generally
$S(P^{p^n})=p^nS(P)$.
By the same argument as above, the elements of
\[
A\cap T_{-p^n\mathbf n_1} B\cap \cdots\cap T_{-p^n\mathbf n_k} B
\]
are not annihilated by $P^{p^n}$ and therefore this intersection is empty in $P^{\perp}$.
We conclude that
\[\mu(T_{-m\mathbf n_0}A \cap T_{-m\mathbf n_1} B\cap \cdots\cap T_{-m\mathbf n_k} B)=0\] 
if $m=p^n$. It follows that $S(P)$ is non-mixing.
We have only used that $P$ annihilates $P^\perp$ in this argument. 
The same proof applies to any $Q\in\langle P\rangle$
since $\langle P\rangle$ is the annihilator of~$P^\perp$.
\end{proof}

The following deep result of Masser can be found in~\cite[Cor 3.9]{Kitchens-Schmidt-1993}.

\begin{thm}\label{masser}
Let $P$ be an irreducible Laurent polynomial. 
A primitive set $E$ is non-mixing for
$P^{\perp}$ if and only if there exist $a,b,\ell \in\mathbb N$ and a 
Laurent polynomial $Q\in \mathbb F_{p^\ell}[X_1^{\pm 1},\ldots,X_d^{\pm 1}]$
with $S(Q)\subset E$, such that $P^{(a)}$ and $Q^{(b)}$ have a common factor.
\end{thm}

\begin{cor}\label{ca-shift-is-mixing}
Let $P^{\perp}$ be a linear cellular automaton shift with $P=X_d-\Phi(X_1,\ldots,X_{d-1})$. 
If all elements $(n_1,\ldots,n_d)\in S$ have zero final coordinate $n_d=0$
then $S$ is mixing.  
\end{cor}

\begin{proof}
By contradiction.
Suppose that $S$ is not mixing. Then $S-\mathbf n$ is not mixing for any $\mathbf n\in S$.
Since $S-\mathbf n$ is primitive, there exists a $Q$ with shape $S(Q)\subset S-\mathbf n$
such that $Q^{(b)}$ and $P^{(a)}$ have a common factor $R$ for some $a,b$. 
$Q$ is a polynomial with indeterminates $X_1,\ldots,X_{d-1}$.
Multiplying by a monomial, if necessary, we may assume that $R$ is
also a polynomial with indeterminates $X_1,\ldots,X_{d-1}$.
Since $R$ divides $P^{(a)}$, which has unit coefficient at $X_d^a$,
we conclude that $R$ divides $1$:
it is a unit. Therefore $Q^{(b)}$ and $P^{(a)}$ have no common factor.
We conclude that $S$ is mixing.	
\end{proof}

\section{Topological rigidity of linear cellular automaton shifts}

We study equivariant maps between linear cellular automaton shifts.
An equivariant map $\phi$ between $G$-systems does not necessarily preserve $0$. 
However, for the systems that we consider here, it is possible to
replace an equivariant map
that does not preserve $0$ by one
that does, as follows.
Since $0$ is fixed so is $\phi(0)$. Therefore the coordinates of $\phi(0)$ are a
constant $c\in\mathbb F_p\setminus\{0\}$.
In other words, $\phi(0)=(c)$, the Laurent series which has coefficient~$c$ at every monomial. 
Since $Q$ annihilates $(c)$, we conclude $Q^{\perp}$ is closed under subtraction of $(c)$. The map $\bar\phi(x)=\phi(x)-(c)$
is equivariant and preserves $0$. 
This shows that the following definition
is not restrictive. 

\begin{definition}
We say that a $\mathbb Z^d$-equivariant map $\phi\colon P^{\perp}\to Q^{\perp}$
is a \textit{homomorphism} if $\phi(0)=0$. If $\phi$ is continuous, then we
say that it is a \textit{topological homomorphism}. If $\phi$ is a module homomorphism,
then we say that it is an \textit{algebraic homomorphism}.
\end{definition}

Let $S$ be a shape. We say that a map $f\colon S\to \mathbb F_p$ is a \textit{configuration},
or, more specifically, an $S$-\textit{configuration}.
Configurations are a higher-dimensional analogue
of words in shift dynamical systems. If $x\in P^{\perp}$, then $x|_S$ represents the configuration
defined by $s\mapsto x_s$.  
We say that this configuration \textit{occurs}
in $P^{\perp}$. Let $\mathcal L_S(P^\perp)$ be the set of all $S$-configurations which
occur in $P^{\perp}$. 
Since $P^{\perp}$ is a group, $\mathcal L_S(P^{\perp})$ is a group.
A map $\gamma\colon \mathcal L_S(P^{\perp})\to\mathbb F_p$
is called a \textit{coding}, or also a \textit{local rule}. The map
\[
(x)_{\mathbf n\in \mathbb Z^d}\to (\gamma(x|_{\mathbf n+S}))_{\mathbf n\in\mathbb Z^d}
\]
is called a \textit{sliding block code}. It is customary to select a shape
of the form $\{0,\ldots,m\}^d\subset \mathbb Z^d$, i.e., a block, hence the name.
Let $0_S$ denote the zero configuration. 
A topological homomorphism corresponds to a sliding block code such that
$\gamma(0_S)=0$.
The topological homomorphism is a
group homomorphism if and only if $\gamma\colon \mathcal L_S(P^{\perp})\to\mathbb F_p$
is a group homomorphism.

\begin{lem}\label{scaling-converse} 
Let $P^{\perp}$ be a linear cellular automaton shift on a finite field 
$\mathbb F_p$. 
Suppose that $\{\mathbf n_1,\ldots,\mathbf n_k\}\subset \mathbb Z^d$
all have final coordinate equal to zero,
and that $C_1,\ldots, C_k\in\mathcal L_S(P^{\perp})$ for some
shape $S$.
Then there exists $x\in P^{\perp}$ and $m\in \N$  such that 
$x|_{S+p^m\mathbf n_i}=C_i$ for $1\leq i\leq k$.
\end{lem}

\begin{proof}
Let $A_i\subset P^{\perp}$ be the cylinder set of all elements
such that $x_S=C_i$.
By Corollary~\ref{ca-shift-is-mixing}, the shift is mixing
on $\{\mathbf n_1,\ldots,\mathbf n_k\}$.
Therefore, for large enough $n\in\N$ 
\[T_{-n\mathbf n_1}A_1\cap\cdots\cap T_{-n\mathbf n_k} A_k\]
has non-zero measure, which implies that this intersection is non-empty.
Any element in this intersection satisfies $x|_{S+n\mathbf n_i}=C_i$. 
We can choose $n$ to be a power of $p$. 
\end{proof}

\begin{lem}\label{finite-group-map}
Let $P^\perp$ and $Q^\perp$ be linear cellular automaton shifts.
If $S(P)\not\subset S(Q)$, then the trivial homomorphism is the only topological
homomorphism between $P^{\perp}$ and $Q^{\perp}$.
\end{lem}

\begin{proof}
Let $\phi\colon P^\perp\to Q^\perp$ be a homomorphism with local rule 
$\gamma\colon \mathcal L_S(P^\perp)\to \mathbb F_p$ for some shape $S$. 
Let $P=X_d-\Phi$ and $Q=X_d-\Psi$
and let $S(\Phi)\cup S(\Psi)=\{\mathbf n_1,\ldots,\mathbf n_k\}$.
By Corollary~\ref{ca-shift-is-mixing}, $\{-\mathbf n_1,\ldots,-\mathbf n_k\}$
 is a mixing set for~$P^\perp$. By Lemma~\ref{scaling-converse}, for any $k$ configurations
$C_i\in\mathcal L_S(P^\perp)$ there exists an $x\in P^\perp$ such that
$x|_{S-p^m\mathbf n_i}=C_i$. Since $S(P)\not\subset S(Q)$, there exists an $\mathbf n_j$
which is in $S(\Phi)$ but not in $S(\Psi)$.  For all $i\not=j$ we take $C_i$ to be the
zero configuration. For $C_j$ we take an arbitrary configuration.

Since $x|_{S-p^m\mathbf n_i}=C_i$, by the local rule $\phi(x)|_{-p^m\mathbf n_i}=\gamma(C_i)$.
All of these configurations except at $C_j$ are zero. In particular $\gamma(C_i)=0$ for all
$i$ such that $\mathbf n_i$ occurs in $\Psi$.
Since $\phi(x)$ is annihilated by $Q^{p^m}=X_d^{p^m}-\Phi^{p^m}$, and since
$S(\Phi^{p^m})=p^mS(\Phi)$, we have that \[\phi(x)_{(0,\ldots,0,-p^m)}=0.\]
In other words, if all of the coefficients that correspond to $\Psi$ are zero, then 
$Q^{p^m}$ can only annihilate $\phi(x)$ if the remaining coefficient corresponding to
$X_d^{p^m}$ is zero as well. 
A similar argument applies to $P^{p^m}$, but here we have that the coefficient
corresponding to $\mathbf n_j$ is not zero. So here the coefficient corresponding
to $X_d^{p^m}$ is non-zero. The configuration at $x|_{S+(0,\ldots,0,-p^m)}$
has to match up against $C_j$. More specifically, if 
$c$ is the coefficient of $\Psi$ at $X^{\mathbf n_j}$, then
\[
x_{S+(0,\ldots,0,-p^m)}=cC_j
\]
By the local rule $\gamma(cC_j)=0$ and since $C_j$ is arbitrary and $c$ is non-zero,
the local rule is trivial.
\end{proof}

A Laurent polynomial is a series
\[
\sum_{\mathbf n\in\Z^d} c_P(\mathbf n)X^\mathbf n 
\]  
in which all but finitely many coefficients $c_P(\mathbf n)$ are zero.
Now think of $c_P$ as a map of $\mathbf Z^d$ into the $1$-configurations.
Extending this, for a shape $S$ let $C_{P,S}$ be a map of $\mathbf Z^d$
into $\mathcal L_S(P^{\perp})$.
For a Laurent polynomial $Q$ we denote 
\[
Q\cdot C_{P,S}=\sum_{\mathbf n\in\Z^d} c_Q(\mathbf n)C_{P,S}(\mathbf n).
\]
Since $\mathcal L_S(P^\perp)$ is a group and since this is a finite sum,
$Q\cdot C_{P,S}$ is a well-defined configuration. 
If $\gamma\colon \mathcal L_S(P^\perp)\to \mathbb F_p$ is a local rule,
then $\gamma\circ C_{P,S}$ is a map $\Z^d\to\mathbb F_p$,
which we denote by $\gamma(C_{P,S})$. 

\begin{lem}
Let $P^\perp, Q^\perp$ be linear cellular automaton shifts
such that $S(P)\subset S(Q)$ and such that
$P=X_d-\Phi$ and $Q=X_d-\Psi$. A topological
homomorphism between $P^{\perp}$ and $Q^{\perp}$ with
local rule $\gamma$ satisfies
the functional equation
\[
\gamma(\Phi\cdot C_{P,S})=\Psi\cdot\gamma(C_{P,S})
\]
\end{lem}

\begin{proof}
Since $S(\Psi)$ is a mixing set there exists $x\in P^{\perp}$
and $m\in\N$ such that $x|_{S-p^m\mathbf n}=C_{P,S}(\mathbf n)$
for $\mathbf n\in S(\Psi)$. 
Since $P^{p^m}\cdot x=0$ we have 
\[x|_{S-(0,\ldots,0,p^m)}=\Phi\cdot C_{P,S}.\]
If $\phi$ is the topological homomorphism with local rule~$\gamma$,
then $\phi(x)(-\mathbf n)=\gamma(C_{P,S}(\mathbf n))$ for $\mathbf n\in S(\Psi)$
and $\phi(x)(0,\ldots,0,-p^m)=\gamma(\Phi\cdot C_{P,S})$.
Since $Q^{p^m}\cdot\phi(x)=0$ we also have
\[\phi(x)|_{(0,\ldots,0,-p^m)}=\Psi\cdot \gamma(C_{P,S}).\]
\end{proof}

\begin{lem}
Let $P^\perp, Q^\perp$ be linear cellular automaton shifts
such that $S(P)\subset S(Q)$. A topological
homomorphism between $P^{\perp}$ and $Q^{\perp}$ 
is an algebraic homomorphism.
\end{lem}

\begin{proof}
We need to show that the local rule $\gamma$ satisfies
$\gamma(C+D)=\gamma(C)+\gamma(D)$ for $C,D\in \mathcal L_S(P^{\perp})$.
Let $\mathbf m,\mathbf n\in S(P)$ and let $c_m,c_n$ be the corresponding
coefficients in $P$, while $d_m,d_n$ are the coefficients in $Q$.
By choosing $C_{S,P}(\mathbf m)=C$ and $C_{S,P}(\mathbf n)=D$ and
all other shapes zero, we obtain that 
\[
\gamma(c_mC+c_nD)=d_m\gamma(C)+d_n\gamma(D)
\]
by the functional equation in the previous lemma. 
Taking the second shape to be zero, we obtain that
$\gamma(c_mC)=d_m\gamma(C)$ for any shape $C$.
Taking the first shape to be zero, we obtain
$\gamma(c_nD)=d_n\gamma(D)$.
Therefore
\[
\gamma(c_mC+c_nD)=\gamma(c_mC)+\gamma(c_nD).
\]
The homomorphism is algebraic.
\end{proof}

\begin{thm}\label{homomorphisms-are-algebraic}
A topological homomorphism between linear cellular automata  
$P^{\perp}$ and $Q^{\perp}$ is algebraic. 
The shifts are topologically
isomorphic if and only if $P^\perp=Q^\perp$.
\end{thm}

\begin{proof}
By the previous lemma, we already know that a topological homomorphism between cellular automaton shifts is algebraic if $S(P)\subset S(Q)$.
We also saw that it is trivial if $S(P)\not\subset S(Q)$.
Therefore, topological homomorphisms are necessarily algebraic.

Let $f\colon M\to N$ be a module isomorphism.
If $r$ annihilates $M$, then $rN=rf(M)=f(rM)=0$ 
and therefore $r$ annihilates $N$. 
Since $f$ is an isomorphism,
the same argument applies to $f^{-1}$. The annihilators are equal,
which in our case means that $\langle P\rangle=\langle Q\rangle $.
For linear cellular automaton shifts this even implies that $P=Q$.
\end{proof}

\begin{cor}\label{homomorphisms-are-trivial}
A homomorphism between linear cellular automata $P^\perp$ and $Q^\perp$ is trivial if $P\not=Q$.
\end{cor} 

\begin{proof}
We will use that the Pontryagin dual of $P^\perp$ is 
\[
\widehat{P^\perp}=
\mathbb F_p[X_1^{\pm 1},\ldots,X_d^{\pm 1}]/\langle P\rangle.
\]
A homomorphism $\phi\colon P^\perp\to Q^\perp$ is algebraic, i.e.,
it is a module homomorphism. 
Its Pontryagin dual $\widehat\phi\colon\widehat{Q^\perp}
\to \widehat{P^\perp}$
is a module homomorphism as well.
It follows that $\widehat\phi(R)=R\widehat\phi(1)$.
Therefore, the module homomorphism
is determined by the value $\widehat\phi(1)$.
Now $Q$ represents $0$ in $\widehat{Q^\perp}$ and therefore
$Q\widehat\phi(1)$ represents $0$ in $\widehat{P^\perp}$.
In other words, $P$ divides $Q\widehat\phi(1)$. Since
$P$ and $Q$ are irreducible and since 
$\mathbb F_p[X_1^{\pm 1},\ldots,X_d^{\pm 1}]$ is a unique
factorization domain, $P$ divides $\widehat\phi(1)$.
The homomorphism is trivial.
\end{proof}

The automorphism group of a one-dimensional
Markov shift is large. In general, it is non-amenable~\cite{Boyle-Lind-Rudolph-1988}.
In contrast, the automorphism group of a linear cellular automaton shift
is small. 

\begin{cor}\label{automorphism-group-is-finitely-generated}
The automorphism group of a linear cellular automaton shift is a finitely
generated abelian group.
\end{cor}

\begin{proof}
Again, it is easier to consider automorphisms of the dual module.
We found that a module homomorphism is a multiplication by a 
Laurent polynomial $R$.
In case the homomorphism is an automorphism, there exists a $Q$ such
that multiplication by $QR$ is equivalent to multiplication by $1$.
In other words, $R$ is a unit in 
$\mathbb F_p[X_1^{\pm 1},\ldots,X_d^{\pm 1}]/\langle P\rangle$, which is
 the localization of the ring
$\mathbb F_p[X_1,\ldots,X_d]/\langle P\rangle$ by the
multiplicative set generated by $\{X_1,\ldots,X_d\}$.
The polynomial $P$ is equal to $X_d-\Phi(X_1,\ldots,X_d)$
and therefore
\[\mathbb F_p[X_1,\ldots,X_d]/\langle P\rangle\cong \mathbb F_p[X_1,\ldots,X_{d-1}].\]
Under this isomorphism, $\mathbb F_p[X_1^{\pm 1},\ldots,X_d^{\pm 1}]/\langle P\rangle$
is the localization of 
$\mathbb F_p[X_1,\ldots,X_{d-1}]$ by the multiplicative
set generated by $\{X_1,\ldots,X_{d-1},\Phi(X_1,\ldots,X_{d-1})\}$.
The units in this ring are of the form $U/V$ for polynomials $U,V$
whose prime factor decomposition contains only $X_j$ for $j=1,\ldots,d-1$,
or primes appearing in $\Phi$.
\end{proof}

The automorphism group of a $G$-system necessarily includes
$G$. For certain 
cellular automata, it does not contain much more than that.

\begin{cor}
Suppose that $\Phi(X_1,\ldots,X_{d-1})$ is irreducible.
Then an automorphism of the corresponding linear cellular automaton $P^\perp$ is a multiplication
by $cX_1^{n_1}\ldots X_d^{n_d}$.
\end{cor}

\begin{proof}
This follows from the proof of the previous lemma
combined with the fact that $\Phi=X_d$ in $\mathbb F_p[X_1^{\pm 1},\ldots,X_d^{\pm 1}]/\langle P\rangle$.
\end{proof}

\section{Concluding remarks}

Our topological rigidity result is a weak version of
the measure rigidity result in \cite{Kitchens-Schmidt-2000}. 
The proof in that paper depends on the relatively straightforward
algebraic characterization of $2$-mixing.
Our proof depends on
Masser's algebraic characterization of $k$-mixing, which is a
much deeper result. In our topological setting, we produce a weaker
result using stronger machinery.

The proof of Masser's theorem in~\cite{Masser-2004}
depends on ideas from algebraic geometry.
It has recently been extended in~\cite{Derksen-Masser-2012,Derksen-Masser-2015},
presenting an efficient method to compute minimal non-mixing sets.
Another approach to computing non-mixing sets can be found in~\cite{Bergelson-2018}.

\section{Acknowledgement}

We would like to thank Michael Baake  and Marcus Pivato for useful discussions. 
Reem Yassawi was supported by STAR travel grant~323600
and by the European Research Council (ERC) under the European Union's Horizon 2020 research and innovation programme, Grant Agreement No 648132.  
She also thanks  IRIF, Universit\'e Paris Diderot-Paris 7, for its hospitality and support.

\end{document}